\documentclass[11pt,a4paper]{article}

\usepackage[headinclude]{typearea}
\usepackage[english]{babel}           
\usepackage[utf8]{inputenc}
\usepackage[T1]{fontenc}
\usepackage{scrtime}
\usepackage{srcltx}
\usepackage{amsmath,amsthm,amsfonts,amssymb}
\usepackage {color}                                  
\usepackage {pifont}                                 
\usepackage {multicol,multirow}                               
\usepackage[numbers]{natbib}                                                     
\usepackage[normalem]{ulem}                                                     
\usepackage [scaled=.90]{helvet}                     
\usepackage {courier}                                
\usepackage {graphicx}                               
\usepackage {array}                                  
\usepackage {longtable}                              
\usepackage{fancyvrb}
\usepackage{enumerate} 
\usepackage{ifpdf}
\usepackage{caption}
\usepackage{mathrsfs}

\usepackage{setspace}

\usepackage{marvosym}

\theoremstyle{plain}
\newtheorem{theorem}{Theorem}[section]

\newtheorem{lemma}[theorem]{Lemma}

\theoremstyle{definition}
\newtheorem{assumption}{Assumption}[section]
\newtheorem{definition}[theorem]{Definition}
\newtheorem{example}[theorem]{Example}

\newcommand{\exclude}[1]{}

\newcommand{\proj}{p}

\newcommand{\err}{{\operatorname{err}}}
\newcommand{\id}{{\operatorname{id}_{\R^d}}}
\newcommand{\tang}{{\operatorname{tang}}}
\newcommand{\sdist}{D}
\newcommand{\idz}{{\operatorname{id}_{\R^2}}}

\newcommand{\1}{{\mathbf{1}}}
\newcommand{\E}{{\mathbb{E}}}

\newcommand{\N}{{\mathbb{N}}}
\renewcommand{\P}{{\mathbb{P}}}

\newcommand{\R}{{\mathbb{R}}}

\newcommand{\strongrate}{1/4-\epsilon}

\definecolor{darkgreen}{rgb}{0,0.5,0}
\definecolor{lightgreen}{rgb}{0.5,0.9,0.5}
\definecolor{magenta}{rgb}{0.75,0,0.25}

\definecolor{violet}{rgb}{0.25,0,0.75}

\newcommand{\hypsurf}{\Theta}
\newcommand{\appx}{X^\delta}
\newcommand{\appz}{Z^\delta}
\newcommand{\es}{{\underline{s}}}
\newcommand{\set}[1]{{\Omega_{\delta,\varepsilon,#1}}}
\newcommand{\setc}[1]{{\Omega^c_{\delta,\varepsilon,#1}}}
\newcommand{\tr}{\operatorname{tr}}
\newcommand{\setS}{S}
\newcommand{\drift}{A}
\newcommand{\diff}{B}
\newcommand{\smallconsti}{\left(\frac{4}{3 c_0}\right)}

\newcommand{\divthr}{f}
\newcommand{\divsig}{\beta}
\newcommand{\divmu}{\vartheta}
\newcommand{\dsig}{a}
\newcommand{\othereps}{\epsilon}

\renewcommand{\P}{{\mathbb P}}

\newcommand{\cS}{{\cal S}}

\newcommand{\be}{\begin{equation}}
\newcommand{\ee}{\end{equation}}
\newcommand{\bea}{\begin{eqnarray}}
\newcommand{\eea}{\end{eqnarray}}
\newcommand{\beast}{\begin{eqnarray*}}
\newcommand{\eeast}{\end{eqnarray*}}
\newcommand{\bproof}{\begin{proof}}
\newcommand{\eproof}{\end{proof}}

\title{Convergence of the Euler-Maruyama method for multidimensional SDEs with discontinuous drift and degenerate diffusion coefficient}

\author{Gunther Leobacher and Michaela Sz\"olgyenyi}
\date{Corrected version, January 2018}

\begin{document}

\maketitle


\begin{abstract}
We prove strong convergence of order $\strongrate$ for arbitrarily small $\othereps>0$ of the Euler-Maruyama method for multidimensional
stochastic differential equations (SDEs) with discontinuous drift and degenerate diffusion coefficient. 
The proof is based on estimating the difference between the Euler-Maruyama scheme and another numerical method, which is constructed by applying the Euler-Maruyama scheme to a transformation of the SDE we aim to solve.\\

\end{abstract}


\section{Introduction}

We consider time-homogeneous stochastic differential equations (SDEs) of the form
\begin{align}\label{eq:sde}
dX_t=\mu(X_t) dt + \sigma(X_t) dW_t\,, \quad X_0=x\,,
\end{align}
where $x\in \R^d$ is the initial value, $\mu:\R^d \longrightarrow \R^d$ is the drift and $\sigma: \R^d \longrightarrow \R^{d \times d}$ is the diffusion coefficient.

The Euler-Maruyama approximation with step-size $\delta>0$ of the solution to \eqref{eq:sde} is given by
\begin{align}\label{eq:euler}
\appx_t=x+\int_0^t \mu(\appx_\es) ds + \int_0^t \sigma(\appx_\es) dW_s\,,
\end{align}
with $\es=j\delta$ for $s\in[j \delta , (j+1)\delta )$, $j=0,\dots,(T-\delta)/\delta$.
In particular, for $t\in\{j \delta: j=0,\dots,(T-\delta)/\delta\}$, we have
\begin{align*}
\appx_{t+\delta}=\appx_{t}+\mu(\appx_{t}) \delta + \sigma(\appx_{t}) 
(W_{t+\delta}-W_{t})\,.
\end{align*}

For $\mu,\sigma$ Lipschitz, \citet{ito1951} proved existence and uniqueness of the solution of \eqref{eq:sde}.
In this case the Euler-Maruyama method \eqref{eq:euler} converges with strong order $1/2$ to the true solution,
see \cite[Theorem 10.2.2]{kloeden1992}. Higher order algorithms exist, but require stronger conditions on the coefficients.\\

In applications, frequently SDEs with less regular coefficients appear.
For example in stochastic control theory, whenever the optimal control is of bang-bang type, meaning that the strategy is of the form $\1_\cS(X)$ for a measurable set $\cS\subseteq \R^d$, the drift of the controlled dynamical system is discontinuous.
Furthermore, there are models which involve only noisy observations of a signal that 
has to be filtered. After applying filtering theory the diffusion coefficient
typically is degenerate in the sense that $\|\sigma(x)^\top v\|=0$, for some $x,v\in\R^d$.
This motivates the study of SDEs with these kind of irregularities in the coefficients.\\

If $\mu$ is bounded and measurable, and $\sigma$ is bounded, Lipschitz, and 
uniformly non-degenerate, i.e.~if there
exists a constant $c_0>0$ such that for all $x\in \R^d$ and all $v\in \R^d$ it holds that $\|\sigma(x)^\top v\|\ge c_0 \|v\|$, \citet{zvonkin1974} and
\citet{veretennikov1981,veretennikov1982} 
prove existence and uniqueness of a solution.
\citet{veretennikov1984} extends these results by allowing a part of the diffusion to be degenerate.

In \cite{sz14} existence and uniqueness of a solution for the case where the drift is
discontinuous at a hyperplane, or a special hypersurface and where the diffusion coefficient is degenerate is proven,
and in \cite{sz2016a} it is shown how these results extend to the non-homogeneous case.\\

Currently, research on numerical methods for SDEs with irregular coefficients is highly active.
\citet{hutzenthaler2012} introduce the tamed Euler-Maruyama scheme and prove strong order $1/2$ convergence for SDEs with continuously differentiable and polynomially growing drift that satisfy a one-sided Lipschitz condition.
\citet{sabanis2013} proves strong convergence of the tamed Euler-Maruyama scheme from a different perspective and also considers the case of locally Lipschitz diffusion coefficient.
\citet{gyongy1998} proves almost sure convergence of the Euler-Maruyama scheme for the case where the drift satisfies a monotonicity condition.

\citet{halidias2008} show that the Euler-Maruyama scheme converges strongly for SDEs with a discontinuous monotone drift coefficient.
\citet{kohatsu2013} show weak convergence with rates smaller than 1 of a method where they first regularize the discontinuous drift and then apply the Euler-Maruyama scheme.
\citet{etore2013,etore2014} introduce an exact simulation algorithm for one-dimensional SDEs with a drift coefficient which is discontinuous in one point, but differentiable everywhere else.
For one-dimensional SDEs with piecewise Lipschitz drift and possibly degenerate diffusion coefficient, in \cite{sz15} an existence and uniqueness result is proven, and a numerical method, which is based on applying the Euler-Maruyama scheme to a transformation of \eqref{eq:sde}, is presented. This method converges with strong order $1/2$.
In \cite{sz2016b} a (non-trivial) extension of the method is introduced, which converges with strong order $1/2$ also in the multidimensional case. The paper also contains an existence and uniqueness result for the multidimensional setting under more general conditions than, e.g., the ones stated in \cite{sz14}.

The method introduced in \cite{sz2016b} is the first numerical method that is proven to converge with positive strong rate for multidimensional SDEs with discontinuous drift and degenerate diffusion coefficient. It requires application of a transformation and its numerical inverse in each step, which makes the method rather slow in practice.
Furthermore, the method requires specific inputs about the geometry of the discontinuity of the drift to calculate this transformation. This is a drawback, if, e.g., the method shall be applied for solving control problems, since the control is usually not explicitly known.
So a method is preferred that can deal with the discontinuities in the drift automatically.

First results in this direction are contained in a series of papers by Ngo and Taguchi. 
In \cite{ngo2016c} they show convergence of order up to $1/4$ of the Euler-Maruyama method for multidimensional SDEs with discontinuous bounded drift that satisfies a one-sided Lipschitz condition and with H\"older continuous, bounded, and uniformly non-degenerate diffusion coefficient.
In \cite{ngo2016a} they extend this result to cases where the drift is not necessarily one-sided Lipschitz for one-dimensional SDEs,
and in \cite{ngo2016b} they extend the result for one-dimensional SDEs by allowing for discontinuities also in the diffusion coefficient.
For many applications, their results fail to be applicable, since they only hold for one-dimensional SDEs and their method of proof relies on uniform non-degeneracy of the diffusion coefficient.

Contrasting the above, there are several delimiting results which
state that even equations with infinitely often differentiable coefficients
cannot always be solved approximately in finite time, even if the
Euler-Maruyama method converges,
cf.~\citet{hairer2015,jentzen2016,muellergronbach2016,yaroslavtseva2016}.
However, there is still a big gap between the assumptions on the coefficients
under which convergence with strong convergence rate has been proven and the
properties of the coefficients of the equation presented in
\cite{hairer2015}.\\

In this paper we prove strong convergence of order $\strongrate$ for arbitrarily small $\othereps>0$ of the Euler-Maruyama method for multidimensional SDEs with discontinuous drift satisfying a piecewise Lipschitz condition and with a degenerate diffusion coefficient.
Note that we do not impose a one-sided Lipschitz condition on the drift.
So even for SDEs with non-degenerate diffusion coefficient, which do not have a one-sided Lipschitz drift, this result is novel.

Our convergence proof is based on estimating the difference between the Euler-Maruyama scheme and the scheme presented in \cite{sz2016b}.
Close to the set of discontinuities of the drift, we have no tight estimate of this difference, so we
need to study the occupation time of an It\^o process with degenerate diffusion coefficient there.
Away from the set of discontinuities, it is essential to estimate the probability that during one step the distance between the interpolation of the Euler-Maruyama method and the previous Euler-Maruyama step becomes greater than some threshold.

This paper's result is the first one that gives strong convergence and also a strong convergence rate of a fully explicit scheme for multidimensional SDEs with discontinuous drift and degenerate diffusion coefficient, and the first one for multidimensional SDEs with discontinuous drift that does not satisfy a one-sided Lipschitz condition.

\section{Preliminaries}
\label{sec:preliminaries}

In this section we first state the assumptions on the coefficients of SDE \eqref{eq:sde},
under which the result of this paper is proven,
then we study the occupation time of an It\^o process close to a hypersurface, and
finally we recall the transformation from \cite{sz2016b}, which is also essential for our proof. 

\subsection{Definitions and assumptions}

We want to prove strong convergence of the Euler-Maruyama method for 
SDEs with discontinuous drift coefficient. Instead of the usual requirement of 
Lipschitz continuity we only assume that the drift is a piecewise Lipschitz
function on the $\R^d$. 

\begin{definition}[\text{\cite[Definitions 3.1 and 3.2]{sz2016b}}]
Let $A\subseteq \R^d$.
\begin{enumerate}
\item For a continuous curve $\gamma:[0,1]\longrightarrow \R^d$, let  $\ell(\gamma)$ denote its length,
\[
\ell(\gamma)=\sup_{n,0\le t_1<\ldots<t_n\le 1}\sum_{k=1}^n \|\gamma(t_k)-\gamma(t_{k-1})\|\,.
\] 
The {\em intrinsic metric} $\rho$ on $A$ 
is given by 
\[
\rho(x,y):=\inf\{\ell(\gamma):\gamma:[0,1]\longrightarrow A \text{ is a continuous curve satisfying } \gamma(0)=x,\, \gamma(1)=y\}\,,
\]
where $\rho(x,y):=\infty$, if there is no continuous curve from $x$
to $y$.  
\item Let $f:A\longrightarrow \R^m$ be a function.
We say that $f$ is {\em intrinsic Lipschitz}, if it is Lipschitz w.r.t. the
intrinsic metric on $A$, i.e. if there exists a constant $L$ such that
\[
\forall x,y\in A: \|f(x)-f(y)\|\le L \rho(x,y)\,.
\]
\end{enumerate}
\end{definition}

The prototypical examples for intrinsic Lipschitz function are given,
like in the one-dimensional case, by differentiable functions with bounded 
derivative.

\begin{lemma}[\text{\cite[Lemma 3.8]{sz2016b}}]\label{lem:diff-lip}
Let $A\subseteq \R^d$ be open and let 
$f:A\longrightarrow \R^m$ be a differentiable function with $\|f'\|<\infty$.
Then $f$ is intrinsic Lipschitz with Lipschitz constant $\|f'\|$.
\end{lemma}

\begin{definition}[\text{\cite[Definition 3.4]{sz2016b}}]\label{def:pw-lip}
A function $f:\R^d\longrightarrow\R^m$ is {\em piecewise Lipschitz}, if
there exists a hypersurface $\hypsurf$ with finitely many connected 
components and with the property, that
the restriction $f|_{\R^d\backslash \hypsurf}$ is intrinsic Lipschitz.
We call $\hypsurf$ an {\em exceptional set} for $f$,
and we call
\[
\sup_{x,y\in \R^d\backslash\hypsurf}\frac{\|f(x)-f(y)\|}{\rho(x,y)}
\]
the \emph{piecewise Lipschitz constant} of $f$.
\end{definition}

In this paper $\hypsurf$ will be a fixed $C^3$-hypersurface, and we will 
only consider piecewise Lipschitz functions with exceptional set $\hypsurf$.
In the following, $L_f$ denotes the piecewise Lipschitz constant of a function $f$, if $f$ is piecewise Lipschitz,
and it denotes the Lipschitz constant, if $f$ is Lipschitz.

We define the distance $d(x,\hypsurf)$ between a point $x$ and the hypersurface 
$\hypsurf$ by $d(x,\hypsurf):=\inf\{\|x-y\|:y \in \hypsurf\}$,
and for every $\varepsilon>0$ we define 
$\hypsurf^\varepsilon:=\{x\in \R^d: d(x,\hypsurf)<\varepsilon\}$.

Recall that, since $\hypsurf\in C^3$, for every $\xi\in \hypsurf$ there exists
an open environment $U\subseteq \hypsurf$ of $\xi$ and a continuously differentiable
function $n:U\longrightarrow \R^d$ such that for every $\zeta\in U$ the vector
$n(\zeta)$ has length 1 and is orthogonal to the tangent space of $\hypsurf$ in
$\zeta$.
On a given connected open subset of $\hypsurf$ the local unit normal vector $n$ is unique up to a factor $\pm 1$.

We recall a definition from differential geometry.\\

\begin{definition}\label{def:positivereach}
Let $\hypsurf\in \R^d$ be any set. 
\begin{enumerate}
\item \label{def:ucpp} 
An environment $\hypsurf^\varepsilon$ is said to have the 
{\em unique closest point property}, if for every $x\in \R^d$
with $d(x,\hypsurf)<\varepsilon$ there is a unique $\proj\in \hypsurf$ with
$d(x,\hypsurf)=\|x-\proj\|$.
Therefore, we can define a mapping 
$\proj:\hypsurf^{\varepsilon}\longrightarrow \hypsurf$ assigning to each $x$
the point $\proj(x)$ in $\hypsurf$ closest to $x$.
\item 
$\hypsurf$ is said to be of {\em positive reach}, if there exists  
$\varepsilon>0$ such that  $\hypsurf^\varepsilon$ has the  
unique closest point property.
The {\em reach} of $\hypsurf$ is the supremum 
over all such $\varepsilon$ if such an $\varepsilon$ exists, and 
0 otherwise. 
\end{enumerate}

\end{definition}

Now, we give assumptions which are sufficient for the results in
\cite{sz2016b} to hold and which we need to prove the main result here.

\begin{assumption}\label{ass:existence}
We assume the following for the coefficients of \eqref{eq:sde}:
\begin{enumerate}
\item\label{ass:existence-musigmabounded} $\mu$ and $\sigma$ are bounded;
\item \label{ass:existence-sigmalip} the diffusion coefficient $\sigma$ is Lipschitz;
\item \label{ass:existence-pwlip} the drift coefficient $\mu$ is a piecewise Lipschitz function $\R^d\longrightarrow \R^d$.
Its exceptional set $\hypsurf$ is a $C^4$-hypersurface of positive reach and
every unit normal vector $n$ of $\hypsurf$ has bounded second and third derivative;
\item \label{ass:existence-nonparallelity} {\em non-parallelity condition:}
there exists a constant $c_0>0$ such that $\|\sigma(\xi)^\top n(\xi)\|\ge c_0$ for
all $\xi\in \hypsurf$;
\item \label{ass:existence-alpha} the function $\alpha: \hypsurf \longrightarrow \R^d$ defined by
\begin{align}\label{eq:alphad}
\alpha(\xi):=\lim_{h\to 0+}\frac{\mu(\xi-h n(\xi))-\mu(\xi+hn(\xi))}{2 \|\sigma(\xi)^\top n(\xi)\|^2}
\end{align}
is $C^3$ and all derivatives up to order three are bounded.
\end{enumerate}
\end{assumption}

\begin{theorem}[\text{\cite[Theorem 3.21]{sz2016b}}]
Let Assumption \ref{ass:existence} hold. Then SDE \eqref{eq:sde} has a unique strong solution.
\end{theorem}

\paragraph{Remark on Assumption \ref{ass:existence}:}
\hfill
\begin{enumerate}
 \item For existence and uniqueness of a solution to \eqref{eq:sde}, in \cite[Theorem 3.21]{sz2016b} instead of Assumption \ref{ass:existence}.\ref{ass:existence-musigmabounded} 
only boundedness in an $\varepsilon$-environment of $\hypsurf$ is needed.
However, for the proof of our convergence result we require global boundedness.
Note that other results in the literature on numerical methods for SDEs with discontinuous drift also rely on boundedness of the coefficients, cf.~\cite{ngo2016a,ngo2016b,ngo2016c}.
 \item Assumption \ref{ass:existence}.\ref{ass:existence-sigmalip} is a technical condition; the focus in this paper is on other types of irregularities in the coefficients. There are results in the literature, where the authors deal with a non-globally Lipschitz diffusion coefficient, see, e.g., \cite{gyongy2011}, but in contributions where only H\"older continuity is required for $\sigma$, usually uniform non-degeneracy is assumed.
 \item Assumption \ref{ass:existence}.\ref{ass:existence-pwlip} is a geometrical condition which we require in order to locally flatten $\hypsurf$, i.e.~to map $\hypsurf$ to a hyperplane in a regular way.
 This is crucial in many places in \cite{sz2016b} and here, in particular for the proof of Theorem \ref{th:occtime}  below.
 
  In addition to that, Assumption \ref{ass:existence}.\ref{ass:existence-pwlip} implies that there exists a constant $c_1$ such that $\|n'(\xi)\|\le c_1$ for every $\xi\in\hypsurf$ and every orthonormal vector $n$ on $\hypsurf$, see \cite[Lemma 3.10]{sz2016b}.
 \item Assumption \ref{ass:existence}.\ref{ass:existence-nonparallelity} means that the diffusion coefficient must have a component orthogonal to $\hypsurf$ in all $\xi \in \hypsurf$.
 This condition is significantly weaker than uniform non-degeneracy, and it is essential: in \cite{sz14} we give a counterexample for the case where the non-parallelity condition does not hold. Then, even existence of a solution is not guaranteed.
 \item Assumption \ref{ass:existence}.\ref{ass:existence-alpha} is a technical
condition, which is required for our transformation method to work. Boundedness
of $\alpha$ and $\alpha'$ is needed for proving the local invertibility of our
transform. Existence and boundedness of $\alpha''$ and $\alpha'''$ is required
for the multidimensional version of It\^o's formula to hold for the transform,
see \cite{sz2016b}.  

Moreover, it has been shown in \cite[Proposition 3.13]{sz2016b} that $\alpha$ is
a well-defined function on $\hypsurf$, i.e.~it does not depend on the choice of the normal vector $n$ and, in particular, on its sign.
\end{enumerate}

\begin{example}\label{ex:compact}
Suppose $\hypsurf$ is the finite and disjoint union of orientable compact $C^4$-manifolds.
Then $\hypsurf$ is of positive reach by the lemma in \cite{foote1984}, and each 
connected component of $\hypsurf$ separates the $\R^n$ into two open connected
components by the Jordan-Brouwer separation theorem, see \cite{lima1988}.

Thus $\R^d\backslash\hypsurf$ is the union of finitely many disjoint open
connected subsets of $\R^d$; we can write $\R^d\backslash\hypsurf=A_1\cup\dots\cup A_n$.

Suppose there exist bounded and Lipschitz $C^3$-functions 
$\mu_1,\dots,\mu_n:\R^d\longrightarrow\R^d$
such that $\mu=\sum_{k=1}^n \1_{A_k}\mu_k$, and suppose that 
$\sigma:\R^d\longrightarrow\R^{d\times d}$ is bounded, Lipschitz, and $C^3$ with $\sigma(\xi)^\top n(\xi)\ne 0$
for every $\xi\in \hypsurf$.

Then it is readily checked that $\mu$ and $\sigma$ satisfy Assumption \ref{ass:existence}.
\end{example}

In Section \ref{sec:examples} we present a number of concrete 
examples which satisfy Assumption \ref{ass:existence} and we perform numerical
tests on the associated SDEs.

\subsection{Occupation time close to a hypersurface}
\label{subsec:occhypsurf}

In this section we study the occupation time of an It\^o process
close to a $C^3$-hypersurface.
In the proof of our main theorem, the Euler-Maruyama approximation 
$\appx$ in equation 
\eqref{eq:euler} will play the role of that It\^o
process.

\begin{theorem}\label{th:occtime} 
Let $\hypsurf$ be a $C^3$-hypersurface of positive reach  and
let ${\varepsilon_0}>0$ be such that the closure of $\hypsurf^{\varepsilon_0}$ has the unique closest point property.
Let further
$X=(X_t)_{t\ge 0}$ be an $\R^d$-valued It\^o process
\[
X_t=X_0+\int_0^t \drift_s ds+\int_0^t \diff_s dW_s\,,
\]
with progressively measurable processes $\drift=(\drift_t)_{t\ge 0}$, $\diff=(\diff_t)_{t\ge 0}$, where $\drift$ is $\R^d$-valued and 
$\diff$ is $\R^{d\times d}$-valued. 
Let the coefficients $\drift,\diff$ be such that
\begin{enumerate} 
\item \label{it:bounded-coeff}
there exists a constant $c_{AB}$ such that 
for almost all $\omega\in \Omega$ it holds that
\[
\forall t\in[0,T]: X_t(\omega)\in \hypsurf^{\varepsilon_0} \Longrightarrow 
\max(\|\drift_t(\omega)\|,\|\diff_t(\omega)\|)\le c_{AB}\,;
\] 
\item\label{it:orth-abstract}
there exists a constant $c_0$ such that 
for almost all $\omega\in \Omega$ it holds that
\[
\forall t\in[0,T]: X_t(\omega)\in \hypsurf^{\varepsilon_0} \Longrightarrow 
n (p(X_t(\omega)) )^\top\diff_t(\omega)\diff_t(\omega)^\top n (p(X_t(\omega) ))\ge c_0\,.
\] 
\end{enumerate}
Then there exists a constant $C$ such that for all 
$0<\varepsilon<\varepsilon_0/2$,
\[
\int_0^{T} \P \left( \{X_s \in \hypsurf^\varepsilon \} \right) ds
\le C\varepsilon \,.
\]
\end{theorem}

For the proof we will construct a one-dimensional It\^o process $Y$ with 
the property that $Y$ is close to $0$, if and only if $X$ is close to $\hypsurf$.
For the construction of $Y$ we decompose the path of $X$ into pieces
close to $\hypsurf$ and pieces farther away. These pieces are then 
mapped to $\R$ by using a signed distance of $X$ from $\hypsurf$
and pasted together in a continuous way. 

A signed distance to $\hypsurf$ is locally given by 
$D(x):=n(p(x))^\top(x-p(x))$, where $n$ is a local unit normal vector.

\begin{lemma}\label{lem:signed-distance}
For all $x\in \hypsurf^{\varepsilon_0}$ it holds that $D'(x)=n(p(x))^\top$.
\end{lemma}

\begin{proof}
Fix $x\in \hypsurf^{\varepsilon_0}\backslash \hypsurf$ and  consider the
function $h$ defined by $h(b):=\|x-p(x+b)\|^2$.
By definition of the projection map $p$, $h$ has a minimum in $b=0$,
such that $h'(0)=0$. Hence from $h'(b)=-2(x-p(x+b))^\top p'(x+b)$, 
we get $(x-p(x))^\top p'(x)=0$. This implies $n(p(x))^\top p'(x)=0$,
since $(x-p(x))$ is a scalar multiple of $n(p(x))$.

Using that $D(x)=\dsig \|x-p(x)\|$ for an $\dsig\in \{-1,1\}$, we compute
\begin{align}\label{eq:Dstrich}
\nonumber D'(x)&=\dsig \|x-p(x)\|^{-1} (x-p(x))^\top (\id-p'(x))
=\dsig \,n(p(x))^\top (\id-p'(x))\\
&=\dsig \big(n(p(x))^\top-n(p(x))^\top p'(x)\big)
=\dsig n(p(x))^\top\,.
\end{align}
For $\psi\in \R$ with $|\psi|$ small we get
\begin{align*}
D(x+\psi n(p(x)))&=n\Big(p\big(x+\psi n(p(x))\big)\Big)^\top \Big(x+\psi n(p(x))-p\big(x+\psi n(p(x))\big)\Big)\\
&=n(p(x))^\top (x+\psi n(p(x))-p(x))=D(x)+\psi\,,
\end{align*}
such that the directional derivative of $D$ in direction $n(p(x))$ in $x$ is 1.
From this and from \eqref{eq:Dstrich} it follows that
$D'(x)=n(p(x))^\top$. 
This also holds for $x\in \hypsurf$ by the continuity of $D'$.  
\end{proof}

The following lemma states that for any continuous curve $\gamma$ in $\hypsurf^{\varepsilon_0}$ there is a continuous path of unit normal vectors,
such that to every point of $\gamma$ we can assign a signed distance in 
a continuous way.

\begin{lemma}\label{lem:gamma-path}
Let $\gamma:[a,b]\longrightarrow\hypsurf^{\varepsilon_0}$ be a continuous function.
Then there exists $m:[a,b]\longrightarrow\R^d$ such that 
\begin{enumerate}
\item $m$ is continuous;
\item $\|m(t)\|=1$ for all $t\in [a,b]$;
\item $m(t)$ is orthogonal to $\hypsurf$ in the point $p(\gamma(t))$ for all $t\in [a,b]$.
\end{enumerate}
\end{lemma}

\begin{proof}
For $\xi\in\hypsurf$ we denote the tangent space to $\hypsurf$ 
in $\xi$ by $\tang_\xi$. Let 
\[
\setS:=\{a\le s\le b: \exists m:[a,s]\longrightarrow\R^d\text{ continuous},
\;\|m(t)\|=1,\; m(t)\bot \tang_{p(\gamma(t))}\;\forall t\in [a,s]\}. 
\]
The set $\setS$ is nonempty and its elements are bounded by $b$. Let $s_1:=\sup \setS$.
There exists an open and connected subset $U\subseteq\hypsurf$ such that 
$p(\gamma(s_1))\in U$, and a unit normal vector $n_1:U\longrightarrow \R^d$.

Since $U$ is open and $p\circ \gamma$ is continuous, there exists $\eta>0$
such that $p(\gamma([s_1-\eta,s_1]))\subseteq U$. By the definition of
$s_1$ there exists $s\in (s_1-\eta,s_1)$ and 
$m:[a,s]\longrightarrow \R^d$
continuous, with $\|m(t)\|=1$ and $m(t)\bot \tang_{p(\gamma(t))}$ for all $ t\in [a,s]$.

Since $n_1$ is unique up to a factor $\pm 1$, the mapping 
\(n_1\circ p \circ \gamma\) either coincides  with $m$ or $-m$ on 
$(s_1-\eta,s)$. Without loss of generality we may assume that the former is the case.
Thus we can extend $m$ continuously to $[a,s_1]$ by defining 
$m(t):= n_1(p(\gamma(t)))$ for all $t\in (s,s_1]$.

Now, if $s_1$ was strictly smaller than $b$, then we could use the same mapping
\(n_1\circ p \circ \gamma\) to extend $m$ continuously beyond $s_1$, 
contradicting the definition of $s_1$.
\end{proof}

We will need the following estimate on the local time of a 
one-dimensional It\^o process.

\begin{lemma}\label{lem:local-time}
Let $Y=(Y_t)_{t\ge 0}$ be an It\^o process with bounded 
and progressively measurable coefficients $\hat \drift=(\hat \drift_t)_{t\ge0},\hat \diff=(\hat \diff_t)_{t\ge0}$.

Then $\sup_{y\in \R}\E(L_T^y(Y))
\le \left(3 T^2\|\hat \drift\|_\infty^2 +\frac{3}{2}T \|\hat \diff_s\|_\infty^2\right)^{1/2}$.
\end{lemma}

The claim is a special case of \cite[Lemma 3.2]{ngo2016a}. 
We give a proof for the convenience of the reader.

\begin{proof}
From the Meyer-Tanaka formula \cite[Section 3.7, Eq.~(7.9)]{karatzas1991} we have
\begin{align*}
2 L_T^y(Y)&=|Y_T-y|-|Y_0-y|-\int_0^T \left(\1_{\{Y_s>y\}}-\1_{\{Y_s<y\}}\right)dY_s\\
&\le|Y_T-Y_0|+\left|\int_0^T \left(\1_{\{Y_s>y\}}-\1_{\{Y_s<y\}}\right)dY_s\right|\\
&\le\left|\int_0^T \hat \drift_s ds\right|+ \left|\int_0^T \hat \diff_s dW_s\right|
+\left|\int_0^T \left(\1_{\{Y_s>y\}}-\1_{\{Y_s<y\}}\right)\hat \drift_sds\right|
+\left|\int_0^T \left(\1_{\{Y_s>y\}}-\1_{\{Y_s<y\}}\right)\hat \diff_s dW_s\right|\\
&\le2\int_0^T |\hat \drift_s |ds+ \left|\int_0^T \hat \diff_s dW_s\right|
+\left|\int_0^T \left(\1_{\{Y_s>y\}}-\1_{\{Y_s<y\}}\right)\hat \diff_s dW_s\right|\,.
\end{align*}
Using the inequality $(a+b+c)^2\le 3(a^2+b^2+c^2)$ we get 
\[
4 L_T^y(Y)^2\le 12 \|\hat \drift\|_\infty^2 T^2+3\left|\int_0^T \hat \diff_s dW_s\right|^2 
+3\left|\int_0^T \left(\1_{\{Y_s>y\}}-\1_{\{Y_s<y\}}\right)\hat \diff_s dW_s\right|^2\,,
\]
and, using It\^o's $L^2$-isometry,
\[
4\E\left( L_T^y(Y)^2\right)\le 12 \|\hat \drift\|_\infty^2 T^2+6\int_0^T \hat \diff_s^2 ds\le 12 \|\hat \drift\|_\infty^2 T^2+6T \|\hat \diff\|_\infty^2 \,.
\]
The claim now follows by applying the Cauchy-Schwarz-inequality and taking the supremum over all $y\in \R$.
\end{proof}

We are ready to prove the result of this section.

\begin{proof}[Proof of Theorem \ref{th:occtime}]
Let $\varepsilon_1=\varepsilon_0/2$.
Define a mapping $\lambda: \R \longrightarrow \R$ by 
\[
\lambda(z)= \begin{cases}
z- \frac{2}{3 \varepsilon_1^2}z^3 +\frac{1}{5 \varepsilon_1^4}z^5
& |z|\le \varepsilon_1\\
\frac{8 \varepsilon_1}{15} & z> \varepsilon_1\\
-\frac{8 \varepsilon_1}{15} & z< -\varepsilon_1\,.
\end{cases}
\]
Note that $\lambda'(0)=1$ and $\lambda'(\pm\varepsilon_1)=\lambda''(\pm\varepsilon_1)=0$,
so that $\lambda\in C^2$.

Next we decompose the path of $X$: 
let $\tau_{0}:=\inf\{t\ge 0: X_t\in \hypsurf^{\varepsilon_1}\}$. In particular
we have $\tau_0=0$, if $X_0\in \hypsurf^{\varepsilon_1}$.
For $k\in \N_0$, define 
\begin{align*}
\kappa_{k+1}&:=\inf\{t\ge \tau_{k}: X_t\notin \hypsurf^{2\varepsilon_1}\}\wedge T\,,\\
\tau_{k+1}&:=\inf\{t\ge \kappa_{k+1}: X_t\in \hypsurf^{\varepsilon_1}\}\wedge T\,.
\end{align*}
By Lemma \ref{lem:gamma-path} there exist continuous
$m_k:[\tau_k,\kappa_{k+1}]\longrightarrow \R^d$, with $\|m_k(t)\|=1$ and
$m_k(t)\bot \tang_{p(X_t)}$ for all $t\in [\tau_k,\kappa_{k+1}]$.
Without loss of generality $m_0$ can be chosen such that 
$m_0(\tau_0)^\top  (X_{\tau_0}-p(X_{\tau_0}) )\ge 0$.
We construct a one-dimensional process  $Y$ as follows:
\begin{align*}
Y_{t}&=
\begin{cases}
\lambda(m_0(\tau_0)^\top (X_{\tau_0}-p(X_{\tau_0})))
& t\le\tau_0 \\
\lambda(m_k(t)^\top (X_{t}-p(X_{t})))
& t\in [\tau_k,\kappa_{k+1}] \\
\lambda(m_k(\kappa_k)^\top (X_{\kappa_k}-p(X_{\kappa_k})))
& t\in [\kappa_k,\tau_{k}]\,,
\end{cases} 
\end{align*}
where without loss of generality the $m_k$ are chosen such that 
\begin{align}\label{eq:orient1}
\lambda(m_{k+1}(\tau_{k+1})^\top (X_{\tau_{k+1}}-p(X_{\tau_{k+1}})))
=\lambda(m_k(\kappa_k)^\top (X_{\kappa_k}-p(X_{\kappa_k})))\,.
\end{align}
Note that by construction both sides of \eqref{eq:orient1} can only 
take the values $\pm\lambda(\varepsilon_1)$.

We have thus constructed a continuous 
$[\lambda(-\varepsilon_1),\lambda(\varepsilon_1)]$-valued 
process $Y$ with the property that the occupation time of $Y$ in an
environment of $0$ is the same as the occupation time of $X$ in an environment of $\hypsurf$,
i.e.~$Y\in (-\lambda(\varepsilon),\lambda(\varepsilon))$, iff $X\in \hypsurf^\varepsilon$ for all $0<\varepsilon<\varepsilon_1$.

To show that $Y$ is an It\^o process, we want to use It\^o's formula.
For this we recognize that $Y$, depending on its proximity to $\hypsurf$,
is either constant or locally of the form 
$Y_t=\lambda(n(p(X_t))^\top(X_t-p(X_t)))$ for a suitable choice 
of the unit normal vector.
Denote $\sdist(x)=n(p(x))^\top  (x-p(x) )$. The function $\sdist$ is locally
a signed distance to $\hypsurf$ and $D \in C^2$. This can be seen by following
the proof of \cite[Theorem 1]{foote1984}.
Hence, we may apply It\^o's formula to get 
\begin{align*}
dY_t=&\lambda'(\sdist(X_t))\sdist'(X_t) \drift_t dt +\lambda'(\sdist(X_t)) \sdist'(X_t) \diff_t dW_t
+\frac{1}{2}\tr\left(\diff_t^\top\lambda''(\sdist(X_t))\diff_t\right)dt\,.
\end{align*}
By Lemma \ref{lem:signed-distance} we have $\sdist'(x)=n(p(x))^\top$, and hence
\[
(\lambda(\sdist(x)))''=
(\lambda'(\sdist(x))n(p(x))^\top)'
=\lambda''(\sdist(x))n(p(x))n(p(x))^\top+\lambda'(\sdist(x))n'(p(x))\,.
\]
Since $\lambda'$ and $\lambda''$ are bounded by construction,
$\|n(p(x))n(p(x))^\top\|=1$, $\|n'\|$  is bounded
(c.f.~the remark on Assumption \ref{ass:existence}.\ref{ass:existence-pwlip}), and by
Assumption \ref{it:bounded-coeff} of the theorem, the 
coefficients   of $Y$ are uniformly bounded.
Therefore $dY_t=\hat \drift_t dt+ \hat \diff_t dW_t$, with bounded 
and  progressively measurable  $\hat \drift,\hat \diff$.

Let $0<\varepsilon\le \varepsilon_1/2$. For all $|z|\le \varepsilon$, we have
$\lambda'(z)\ge \left(\frac{3}{4}\right)^2$. Thus by Assumption \ref{it:orth-abstract} of the theorem,
\begin{align*}
\left(\frac{3}{4}\right)^2 c_0^2\int_0^t\1_{\left\{X_s\in \hypsurf^\varepsilon\right\}}ds
&= \left(\frac{3}{4}\right)^2 c_0^2\int_0^t\1_{\left\{Y_s\in (-\lambda(\varepsilon),\lambda(\varepsilon))\right\}}ds\\
&\le 
\int_0^t
\1_{\left\{Y_s\in (-\lambda(\varepsilon),\lambda(\varepsilon))\right\}}
\lambda'\left(D(X_s)\right)^2
n(p(X_s))^\top\diff_s\diff_s^\top n(p(X_s))ds\\
&= 
\int_0^t\1_{\left\{Y_s\in (-\lambda(\varepsilon),\lambda(\varepsilon))\right\}}
d\left[Y\right]_s\,.
\end{align*}
By the occupation time formula 
\cite[Chapter 3, 7.1 Theorem]{karatzas1991} for one-dimensional continuous 
semimartingales, we get
\begin{align*}
\int_0^T\P\left(\{X_s\in \hypsurf^\varepsilon\}\right)ds
&\le \smallconsti^{2}
\E\left(\int_0^T\1_{\{Y_s\in (-\lambda(\varepsilon),\lambda(\varepsilon))\}}d\left[Y\right]_s\right)\\
&=2 \smallconsti^{2} \E\left(\int_\R \1_{(-\lambda(\varepsilon),\lambda(\varepsilon))}(y)L_T^{y}\left(Y\right)dy\right)\\
&\le  \frac{4^3}{3^2c_0^2}\, \sup_{y\in \R}\E\left(L_T^{y}\left(Y\right)\right)\varepsilon
\,.
\end{align*}

\end{proof}

\subsection{The transformation}
\label{subsec:transformation}

The proof of convergence is based on a transformation that removes the discontinuity from the drift and makes the drift Lipschitz while preserving the Lipschitz property of the diffusion coefficient.
A suitable transform is presented in \cite{sz2016b}. We recall it here.\\

Define $G:\R^d\longrightarrow\R^d$, 
\[
G(x)=\begin{cases}
 x+\tilde \phi(x) \alpha(\proj(x))&  x\in \hypsurf^{\varepsilon_0}\\
x & x\in \R^d\backslash \hypsurf^{\varepsilon_0}\,,
\end{cases}
\]
where $\varepsilon_0>0$ is smaller than the reach of $\hypsurf$, see 
Assumption \ref{ass:existence}.\ref{ass:existence-pwlip}, 
$\alpha$ is the function defined in Assumption \ref{ass:existence}.\ref{ass:existence-alpha}, and  \begin{align*}
 \tilde \phi(x)=n(\proj(x))^\top(x-\proj(x))
\|x-\proj(x)\|\phi\left(\frac{\|x-\proj(x)\|}{c}\right)\,,                                                                                                                                       
 \end{align*}
with positive constant $c$ and
\begin{align*}
\phi(u)=
\begin{cases}
(1+u)^4(1-u)^4 & |u|\le 1\\
0 & |u|> 1\,.
\end{cases}
\end{align*}

If $c$ is chosen sufficiently small, see \cite[Lemma 3.18]{sz2016b}, $G$ is invertible by \cite[Theorem 3.14]{sz2016b}.
Furthermore, It\^o's formula holds for $G$ and $G^{-1}$ by \cite[Theorem 3.19]{sz2016b}.
By \cite[Lemma 4]{aap-corr}, $G''$ is bounded and piecewise Lipschitz with exceptional set $\Theta$.
Hence, \cite[Lemma 3.6]{sz2016b}, \cite[Lemma 3.8]{sz2016b}, and \cite[Lemma 3.11]{sz2016b} assure that $G'$ is Lipschitz.
\\

With this we can define a process $Z=(Z_t)_{t\ge 0}$ by $Z_t=G(X_t)$, which solves the SDE
\begin{align}\label{eq:SDEtransf}
 dZ_t =\tilde \mu(Z_t) dt+\tilde \sigma(Z_t) dW_t\,,
\end{align}
where
\begin{align*}
\tilde \mu(z)&=G'(G^{-1}(z))\mu(G^{-1}(z))+\frac{1}{2}\tr \left(\sigma(G^{-1}(z))^\top G''(G^{-1}(z))\sigma(G^{-1}(z))\right)\,,\\
\tilde \sigma(z)&=G'(G^{-1}(z)) \sigma(G^{-1}(z)) \,.
\end{align*}
From \cite[Theorem 3.20]{sz2016b} we know that $\tilde\mu$ and $\tilde \sigma$ are
Lipschitz, and hence the solution to \eqref{eq:SDEtransf} can be approximated with strong order $1/2$ using the Euler-Maruyama scheme.

\section{Main result}
\label{sec:result}

We are ready to formulate the main result.

\begin{theorem}\label{thm:conv}
 Let Assumption \ref{ass:existence} hold. Then the Euler-Maruyama method
\eqref{eq:euler} converges to the solution of SDE \eqref{eq:sde} with strong
order $\strongrate$ for arbitrarily small $\othereps>0$, 
i.e.~there exists a
constant $C$ such that for all
$\othereps>0$ it holds that for sufficiently small step size $\delta>0$,
\begin{align*}
  \E\left(\sup_{t\in[0,T]} \|X_t-\appx_t\|^2\right)^{1/2} \le C \delta^{\strongrate}\,.
 \end{align*}
\end{theorem}

In preparation of the proof of the main result, we prove two lemmas.

\begin{lemma}\label{lem:onestep} 
 Let Assumption \ref{ass:existence}.\ref{ass:existence-musigmabounded} hold. Then there exists a constant $C$ such that for sufficiently small step size $\delta$
\begin{align*}
\E \left(\int_0^{T}\|\appx_s -\appx_\es\| ^2 ds\right)  \le C \delta\,.
\end{align*}
\end{lemma}

\begin{proof}
By the definition of the Euler-Maruyama method \eqref{eq:euler} we have
\begin{align*}
 \E \left(\int_0^{T}\|\appx_\es -\appx_s\| ^2 ds\right)
 &  = \sum_{j=0}^{T/\delta-1} \E \left(\int_{j\delta}^{(j+1)\delta}\|\appx_{j \delta} -\appx_s\| ^2 ds\right)\\
 &\le  \frac{T}{\delta}\sup_{t\in \{j \delta: j=0,\dots,T/\delta-1\}} \E \left(\int_{t}^{t+\delta}\|\appx_{t} -\appx_s\| ^2 ds\right)\\
 &  =\frac{T}{\delta}\sup_{t\in \{j \delta: j=0,\dots,T/\delta-1\}} \E \left(\int_{t}^{t+\delta}\|\appx_{t} -\appx_t -\mu(\appx_t) \delta - \sigma(\appx_t)(W_s-W_t)\| ^2 ds\right)\\
 &  \le \frac{ 2T}{\delta} \sup_{t\in \{j \delta: j=0,\dots,T/\delta-1\}}\E \left(\int_{t}^{t+\delta}\|\mu(\appx_t) \delta\|^2 ds + \int_t^{t+\delta} \|\sigma(\appx_t) (W_s-W_t)\| ^2 ds\right)\\
 & \le  \frac{2T}{\delta}  \left(\|\mu\|_\infty^2 \delta^3 + \|\sigma\|^2_\infty \sup_{t\in \{j \delta: j=0,\dots,(T-\delta)/\delta\}} \int_t^{t+\delta}\E( \|W_s-W_t\| ^2) ds\right)\\
&=\frac{2T}{\delta}  \left(\|\mu\|_\infty^2 \delta^3 + d\|\sigma\|^2_\infty \sup_{t\in \{j \delta: j=0,\dots,(T-\delta)/\delta\}}\int_t^{t+\delta}(s-t) ds\right)\\
&=  \frac{2T}{\delta}  \left(\|\mu\|_\infty^2 \delta^3 + \frac{d}{2} \|\sigma\|^2_\infty \delta^2\right) \le C \delta\,.
\end{align*}
\end{proof}

For all $\delta,\varepsilon>0$ and all $j=0,\dots,T/\delta-1$, define
\begin{align}\label{eq:omega}
 \set{j}:=\left\{\omega \in \Omega: \sup_{s\in [j\delta,(j+1)\delta]} 
\left\|\appx_s(\omega)-\appx_\es(\omega)\right\| \ge \varepsilon\right\}\,.
\end{align}

\begin{lemma}\label{lem:omega}
Let Assumption \ref{ass:existence}.\ref{ass:existence-musigmabounded} hold. 
Then there exists a constant $C$ such that for all $0<\delta\le 1$, all $\varepsilon>0$, and all $j=0,\dots,T/\delta-1$, it holds that
$\P(\set{j})\le C\exp(-\varepsilon / \|\sigma\|_\infty\delta^{1/2})$.
\end{lemma}

\begin{proof}
\begin{align*}
\P&\left(\sup_{j\delta\le s\le (j+1)\delta}\|\appx_{s}-\appx_\es\|\ge\varepsilon \right)
=\P\left(\sup_{j\delta\le s\le (j+1)\delta}\|\mu(\appx_\es)(s-\es)+\sigma(\appx_\es)(W_s-W_\es)\|\ge\varepsilon \right)\\
&\le\P\left(\sup_{j\delta\le s\le (j+1)\delta}\Big\{\|\mu\|_\infty \delta + \|\sigma\|_\infty\|W_s-W_\es\|\Big\}\ge\varepsilon \right)
=\P\left(\sup_{j\delta\le s\le (j+1)\delta}\|W_s-W_\es\|\ge\frac{\varepsilon-\|\mu\|_\infty \delta}{\|\sigma\|_\infty} \right)\\
&=\P\left(\sup_{0\le s\le 1}\|W_s-W_0\|\ge\frac{\varepsilon-\|\mu\|_\infty \delta}{\|\sigma\|_\infty\delta^{1/2}} \right)
=\P\left(\exp\left(\sup_{0\le s\le 1}\|W_s\|\right)\ge\exp\left(\frac{\varepsilon-\|\mu\|_\infty \delta}{\|\sigma\|_\infty\delta^{1/2}}\right)\right)\\
&\le \E\left(\exp(\|W_1\|)\right)\exp\left(\frac{\|\mu\|_\infty \delta-\varepsilon}{\|\sigma\|_\infty\delta^{1/2}} \right)
\le C\exp\left(-\frac{\varepsilon}{\|\sigma\|_\infty \delta^{1/2}} \right)\,,
\end{align*}
where we applied Doob's submartingal inequality, and in the last step used that $\delta \le 1$ .
\end{proof}

Now, we are ready to prove our main result.

\begin{proof}[Proof of Theorem \ref{thm:conv}] 
Since $G^{-1}$ is Lipschitz by the proof of \cite[Theorem 3.20]{sz2016b},
\begin{align}\label{eq:est-lip}
\E \Big(\sup_{0\le t\le T} \|X_t-\appx_t \|^2 \Big)^{1/2}
\le L_{G^{-1}} \E \Big(\sup_{0\le t\le T} \|Z_t-G (\appx_t ) \|^2 \Big)^{1/2}\,,
\end{align}
with $Z=G(X)$ as in \eqref{eq:SDEtransf}.
Let $\appz$ be the Euler-Maruyama approximation of $Z$.
It holds that
\begin{align}\label{eq:est-dreieck}
\E \Big(\sup_{0\le t\le T} \|Z_t-G (\appx_t ) \|^2 \Big)^{1/2}
\le \E \Big(\sup_{0\le t\le T} \|Z_t- \appz_t \|^2 \Big)^{1/2}
+\E \Big(\sup_{0\le t\le T} \|\appz_t-G (\appx_t ) \|^2 \Big)^{1/2}\,.
\end{align}
For estimating the first term in \eqref{eq:est-dreieck}, recall that
by \cite[Theorem 3.20]{sz2016b}, the transformed SDE \eqref{eq:SDEtransf} has Lipschitz coefficients.
Since the Euler-Maruyama method converges with strong order $1/2$ for SDEs with Lipschitz coefficients
(see \cite[Theorem 10.2.2]{kloeden1992}), there exists a constant $C_1>0$ such that for sufficiently small $\delta>0$,
\begin{align}\label{eq:est-euler}
\E\Big(\sup_{0\le t\le T}\|Z_t- \appz_t\|^2\Big)\le C_1 \delta\,.
\end{align}

We now turn to the second term in \eqref{eq:est-dreieck}, i.e.~we estimate the difference between $G$ applied to the Euler-Maruyama approximation of $X$ and the Euler-Maruyama approximation of $Z$.
Denote, for all $\tau\in[0,T]$,
\begin{align*}
u(\tau):=\E\left(\sup_{0\le t\le \tau}
\|G(\appx_{t})-\appz_{t}\|^2\right)\,.
\end{align*}
With $\nu(x_1,x_2)=G'(x_1)\mu(x_2)+\frac{1}{2}\tr (\sigma(x_2)^\top G''(x_1)\sigma(x_2))$
we have by It\^o's formula,
\begin{align*}
G(\appx_{t})=G(\appx_{0})+\int_0^{t}\nu(\appx_s,\appx_{\es}) ds+\int_0^{t}G'(\appx_s)\sigma(\appx_{\es})dW_s\,,
\end{align*}
so that
\begin{align*}
u(\tau)
&=\E \left(\sup_{0\le t\le \tau}\left\|\int_0^{t} \nu(\appx_s,\appx_{\es})ds
+\int_0^{t}G'(\appx_s)\sigma(\appx_{\es})dW_s
-\int_0^{t}\tilde \mu\left(\appz_{\es}\right) ds-\int_0^{t}\tilde \sigma\left(\appz_{\es}\right) dW_s\right\|^2\right)\\
&\le \E \left(\sup_{0\le t\le \tau}\left(4\left\|\int_0^{t}\left(\nu(\appx_s,\appx_{\es}) 
-\nu(\appx_{\es},\appx_{\es})\right) ds\right\|^2
+4\left\|\int_0^{t}\left(G'(\appx_s)\sigma(\appx_{\es})-G'(\appx_{\es})\sigma(\appx_{\es})\right)dW_s\right\|^2  \right.\right.
\\ 
&\left.\left.\quad+4\left\|\int_0^{t}\left(\tilde\mu(G(\appx_{\es}))-\tilde \mu(\appz_{\es})\right) ds\right\|^2
+4\left\|\int_0^{t}\left(\tilde\sigma(G(\appx_{\es}))-\tilde \sigma(\appz_{\es})\right)dW_s\right\|^2 \right) \right)\,.
\end{align*}

Applying the Cauchy-Schwarz inequality to the Lebesgue integrals and the $d$-dimensional Burkholder-Davis-Gundy inequality \cite[Lemma 3.7]{hutzenthaler2012} to the It\^o integrals, we obtain
\begin{align}
u(\tau)
&\le 4 T \,\E \left(\int_0^{\tau}\left\|\nu(\appx_s,\appx_{\es}) 
-\nu(\appx_{\es},\appx_{\es})\right\|^2 ds\right)
+8d\,\E\left(\int_0^{\tau}\left\|G'(\appx_s)\sigma(\appx_{\es})-G'(\appx_{\es})\sigma(\appx_{\es})\right\|^2 ds\right) \nonumber
\\
&\quad+4T\,\E\left(\int_0^{\tau}\left\|\tilde\mu(G(\appx_{\es}))-\tilde \mu(\appz_{\es}) \right\|^2 ds\right)
+8d\,\E\left(\int_0^{\tau}\left\|\tilde\sigma(G(\appx_{\es}))-\tilde \sigma(\appz_{\es})\right\|^2 ds\right) \nonumber\\
&=:4T\,E_1+8d\,E_2+4T\,E_3+8d\,E_4\,.\label{eq:g-z}
\end{align}

For estimating $E_1$ in \eqref{eq:g-z}, we will use that
\begin{align*}
\left\|\nu(x_1,x_2)-\nu(x_2,x_2)\right\|^2\le
\begin{cases}
\left(2L_{G'}^2 \|\mu\|_\infty^2  +\frac{1}{2}L_{G''}^2\|\sigma\|_\infty^4\right)\|x_1-x_2\|^2&  x_1\notin \hypsurf^\varepsilon\,,\|x_1-x_2\|<\varepsilon\\
4 \|\mu\|_\infty^2 \|G'\|^2_\infty + \|\sigma\|_\infty^4 \|G''\|_\infty^2 & \text{otherwise}\,.
\end{cases}
\end{align*}

With this and the definition of $\setc{j}$ in \eqref{eq:omega}, we get
\begin{align*}
E_1&=\int_0^{\tau} \E\left(\left\|\nu(\appx_s,\appx_\es)-\nu(\appx_\es,\appx_\es)\right\|^2\left(\1_{\{\appx_s \notin \hypsurf^\varepsilon \}}\1_\setc{\es/\delta}+\1_{\{\appx_s \notin \hypsurf^\varepsilon \}}\1_\set{\es/\delta}+\1_{\{\appx_s \in \hypsurf^\varepsilon \}}\right)\right)ds\\
&\le \left(2 L_{G'}^2  \|\mu\|_\infty^2+\frac{1}{2}L_{G''}^2\|\sigma\|_\infty^4\right)\varepsilon^2 T\\
& \quad + \left(4 \|\mu\|_\infty^2 \|G'\|^2_\infty + \|\sigma\|^4 \|G''\|_\infty^2 \right)
\left(\int_0^{T} \P(\set{\es/\delta}) ds+ \int_0^T\P( \{\appx_s \in \hypsurf^\varepsilon \})ds\right)\,.
\end{align*}

By Lemma \ref{lem:omega}, $\int_0^T \P(\set{\es/\delta}) ds\le C_2\exp(- \varepsilon /\|\sigma\|_\infty\delta^{1/2} )$,
and by Theorem \ref{th:occtime}, $\int_0^T \P( \{\appx_s \in \hypsurf^\varepsilon \})ds \le C_3 \varepsilon$, for suitable constants $C_2,C_3$.
In order to minimize the bound on $E_1$, we choose $\varepsilon$ such that $\exp(- \varepsilon /\|\sigma\|_\infty\delta^{1/2} )+\varepsilon$ is minimized for $\delta$ sufficiently small, yielding $\varepsilon=-\|\sigma\|_\infty \delta^{1/2} \log(\|\sigma\|_\infty \delta^{1/2} )=\|\sigma\|_\infty \delta^{1/2-2 \othereps} (-\delta^{2 \othereps} \log(\|\sigma\|_\infty \delta^{1/2} ))$ for arbitrarily small $ \othereps>0$. Hence, with
$C_4=(2 L_{G'}^2  \|\mu\|_\infty^2+\frac{1}{2}L_{G''}^2\|\sigma\|_\infty^4)T$,
$C_5=(4 \|\mu\|_\infty^2 \|G'\|^2_\infty + \|\sigma\|^4 \|G''\|_\infty^2) C_2$,
$C_6=(4 \|\mu\|_\infty^2 \|G'\|^2_\infty + \|\sigma\|^4 \|G''\|_\infty^2) C_3$,
we get
\begin{align*}
E_1&\le  C_4 \varepsilon^2+C_5 \exp\left(- \frac{\varepsilon}{\|\sigma\|_\infty\delta^{1/2}}\right)
+ C_6 \varepsilon\\
 &=  C_4 \|\sigma\|_\infty^2 \delta^{1-4\othereps} (-\delta^{2 \othereps}\log(\|\sigma\|_\infty \delta^{1/2} ))^2+C_5 \|\sigma\|_\infty \delta^{1/2}+C_6 \|\sigma\|_\infty \delta^{1/2-2 \othereps} (-\delta^{2 \othereps}\log(\|\sigma\|_\infty \delta^{1/2} ))\,.
\end{align*}
Thus, with $C_7=C_4  \|\sigma\|_\infty^2 +C_5 \|\sigma\|_\infty+C_6 \|\sigma\|_\infty$ and for arbitrarily small fixed $\othereps>0$, it holds that for sufficiently small $\delta$
\begin{align}\label{eq:e1}
E_1\le C_7 \delta^{1/2-2\othereps}\,.
\end{align}

For estimating $E_2$ in \eqref{eq:g-z}, we apply Lemma \ref{lem:onestep} to get
\begin{align}\label{eq:e2}
E_2\le L_{G'}^2\|\sigma\|_\infty^2 \int_0^{T}  \E\left(\left\|\appx_s-\appx_{\es}\right\|^2
\right)ds
\le  L_{G'}^2\|\sigma\|_\infty^2  C_8 \delta\,.
\end{align}

For estimating $E_3,E_4$ in \eqref{eq:g-z},
we use that $\tilde \mu, \tilde \sigma$ are Lipschitz by \cite[Theorem 3.20]{sz2016b}, to get
\begin{align}\label{eq:e3}
E_3&\le L_{\tilde\mu}^2\int_0^{\tau}\E\left(\left\|G(\appx_{\es})-\appz_{\es}\right\|^2 \right)ds\le L_{\tilde\mu}^2 \int_0^{\tau}u(s)ds\,,
\\ \label{eq:e4}
E_4& \le L_{\tilde\sigma}^2\int_0^{\tau}\E\left(\left\|G(\appx_{\es})-\appz_{\es}\right\|^2\right) ds\le  L_{\tilde\sigma}^2\int_0^{\tau}u(s)ds\,.
\end{align}

Combining the estimates \eqref{eq:e1},\eqref{eq:e2},\eqref{eq:e3},\eqref{eq:e4} with \eqref{eq:g-z}, we get
\begin{align*}
0\le u(\tau) \le C_9 \int_0^\tau u(s) ds + 4 T C_7 \delta^{1/2-2\othereps}+ 8d L_{G'}^2\|\sigma\|_\infty^2 C_8 \delta\,,
\end{align*}
with
$C_9= 4T L_{\tilde\mu}^2+8dL_{\tilde\sigma}^2$. Using that $4 T C_7 \delta^{1/2-2\othereps}+ 8d L_{G'}^2\|\sigma\|_\infty^2 C_8 \delta\le C_{10}  \delta^{1/2-2\othereps}$ for $\delta\le 1$, and
applying Gronwall's inequality yields for all $\tau \in [0,T]$,
\begin{align}\label{eq:est-gron}
u(\tau)\le C_{10} \exp(C_9 \tau) \delta^{1/2-2\othereps}\,.
\end{align}

Combining \eqref{eq:est-euler} and \eqref{eq:est-gron} with \eqref{eq:est-dreieck}, and the result with \eqref{eq:est-lip}, finally yields
\begin{align*}
\E \left(\sup_{0\le t\le T} \|X_t-\appx_t \|^2 \right)^{1/2}
\le L_{G^{-1}}\Big(C_1 \delta\Big)^{1/2} +L_{G^{-1}}\Big(C_{10}\exp(C_9 T) \delta^{1/2-2\othereps}\Big)^{1/2}
\le C \delta^{1/4-\othereps}\,,
\end{align*}
for a suitably chosen constant $C$, for arbitrarily small $\othereps>0$, and
 for $\delta$ sufficiently small.
\end{proof}

\section{Examples}
\label{sec:examples}

We ran simulations for several examples -- ones of theoretical interest as well as an example coming from applications.

When studying stochastic dynamical systems which include a noisy signal, then filtering this signal leads to a higher dimensional system with a degenerate diffusion coefficient.
Stochastic control problems often lead to an optimal control policy which makes the drift of the system discontinuous.
Examples are models with incomplete market information in mathematical finance where the rate with which cashflows are paid from a firm value process change systematically when the asset-liability ratio passes a certain threshold which then triggers a rating change.

The class of equations studied here appears frequently in several areas of applied mathematics and the natural sciences.

\paragraph{Step-function}
In the first example the drift is the step function
$\mu(x_1,x_2)=(3(\1_{\{x_1 \ge 0\}}-\1_{\{x_1 < 0\}}),1)^\top$, and $\sigma\equiv\idz$.
It can easily be checked that these coefficients satisfy Assumption \ref{ass:existence}.
In particular, note that the non-parallelity condition is trivially satisfied, since $\sigma$ is uniformly non-degenerate. Since $\mu$ does not satisfy a one-sided Lipschitz condition, our result is the first one that gives a strong convergence rate of the Euler-Maruyama method for this example.

\paragraph{Discontinuity along the unit circle}
In this example the drift has a discontinuity along the unit circle,
and the diffusion coefficient is degenerate on the whole of $\R^2$:
\begin{align*}
 \mu(x_1,x_2)=\begin{cases}
(1,1)^\top & x_1^2+x_2^2\ge 1\\
(-x_1,x_2)^\top & x_1^2+x_2^2<1\,,
          \end{cases}
          \qquad
     \sigma(x_1,x_2)=\frac{2}{1+x_1^2+x_2^2}  \left(
  \begin{array}{cc}
   x_1 & 0  \\
   x_2 & 0
    \end{array} \right)\,.
\end{align*}
 Assumption \ref{ass:existence}  largely follows from Example \ref{ex:compact}.  
The non-parallelity condition is readily verified:
\begin{align*} \left\|
 \frac{2}{(1+x_1^2+x_2^2)(x_1^2+x_2^2)}  \left(
  \begin{array}{cc}
   x_1 & x_2  \\
   0 & 0
    \end{array} \right) 
  \left(
  \begin{array}{c}
   x_1  \\
   x_2
    \end{array} \right) 
    \right\|
    =\frac{2\sqrt{x_1^2+x_2^2}}{(1+x_1^2+x_2^2)(x_1^2+x_2^2)} 
    =1
\end{align*}
for all points $(x_1,x_2)$ that lie on the unit circle, i.e.~$x_1^2+x_2^2=1$.

\paragraph{Dividend maximization under incomplete information}
In insurance mathematics, a well-studied problem is the maximization of the expected discounted future dividend payments until the time of ruin of an insurance company, a value which serves as a risk measure.
In \cite{sz16} the problem is studied in a setup that allows for incomplete information about the market.
This leads to a joint filtering and stochastic optimal control problem, and after solving the filtering problem,
the driving dynamics are high dimensional and have a degenerate diffusion coefficient.
This issue is described in more detail in \cite{sz16}.
Solving the stochastic optimal control problem in dimensions higher than three with the usual technique (solving an associated partial differential equation) becomes practically infeasible. Therefore, one has to resort to simulation.
The SDE that has to be simulated has the coefficients

\begin{align*}
\mu(x_1,\dots,x_d)&=
\begin{pmatrix}\divmu_d+\sum_{i=1}^{d-1}(\divmu_i-\divmu_d)x_{i+1}-\bar u\1_{[\divthr(x_2,\dots,x_d),\infty)}(x_1)\\
q_{d1} + \sum_{j=1}^{d-1} (q_{j1}-q_{d1}) x_{j+1}\\
\vdots\\
q_{d(d-1)} + \sum_{j=1}^{d-1} (q_{j(d-1)}-q_{d(d-1)}) x_{j+1}
\end{pmatrix}\,,\\[0.4em]
 \sigma(x_1,\dots,x_d)&=\left(
  \begin{array}{cccc}
   \divsig & 0 & \dots & 0\\
   x_2 \frac{\divmu_1-\divmu_d-\sum_{j=1}^{d-1}(\divmu_j-\divmu_d) x_{j+1}}{\divsig} & \vdots &  & \vdots\\
   \vdots & \vdots &  & \vdots\\
   x_d \frac{\divmu_{d-1}-\divmu_d-\sum_{j=1}^{d-1}(\divmu_j-\divmu_d) x_{j+1}}{\divsig} & 0 & \dots & 0
  \end{array}
 \right)\,,
 \end{align*}
 
where  $\bar u,\divsig,(\divmu_i)_{i=1}^d,(q_{ij})_{i,j=1}^d$ are known constants.
The arguments $x_2,\dots,x_d$ are elements of the simplex $\{(x_2,\dots,x_d)\in[0,1]^{d-1}:\sum_{j=1}^{d-1}x_{j+1}\le1\}$,
and the corresponding processes stay within this simplex almost surely, see \cite{sz16}.
The function $\divthr$ determines the hypersurface $\hypsurf$ along which the drift is discontinuous: $\hypsurf=\{(x_1,\dots,x_d): x_1=\divthr(x_2,\dots,x_d)\}$.
In our simulations we choose $d=5$ and $\divthr$ affine linear, but note that we need not restrict ourselves
to affine linear $\divthr$.\\

We need to check Assumption \ref{ass:existence}: Since $x_2,\dots,x_d \in [0,1]$, $\mu, \sigma$ are bounded,
and all first order derivatives of the entries of $\sigma$ are bounded. Hence, $\sigma$ is Lipschitz.
$\mu$ is piecewise Lipschitz, and since $\divthr$ is affine linear, $\hypsurf \in C^4$.
Whether the non-parallelity condition holds depends on the choice of the parameters,
but for ours the condition is satisfied.
Assumption \ref{ass:existence}.\ref{ass:existence-alpha} can easily be checked.
Note that the coefficients can be extended to the whole of $\R^d$ in a way that they still satisfy our assumptions.

\paragraph{Error estimate}
The $L^2$-error is estimated by
\begin{align*}
\err_k :=  \bar e \,\hat E\left(\left\|X_T^{(k)} - X_T^{(k-1)}\right\|^2\right)^{1/2}\,,
\end{align*}
where $X_T^{(k)}$ is the numerical approximation of $X_T$ with step size
$\delta^{(k)}$, $\hat E$ is an estimator of the mean value using $2^{14}$
paths, and $\bar e$ is a normalizing constant so that $\err_1=\sqrt{1/4}$.

Figure \ref{fig:err} shows $\log_2$ of the estimated $L^2$-error of the Euler-Maruyama approximation of $X_T$ plotted over $\log_2 \delta ^{(k)}$ for the examples presented above.
We observe that the theoretical convergence rate is approximately obtained for the example of a step-function and that the other examples converge at a faster rate.
In particular, for the examples with degenerate diffusion coefficient, the convergence rate is not worse than for the other example.
Even for the step-function example, for sufficiently small step-size the convergence rate seems to be higher than the theoretical one.
Hence, it will be an interesting topic for future research to prove sharpness, or find a sharp bound.\\

\begin{figure}[ht]
\begin{center}
\includegraphics[scale=0.6]{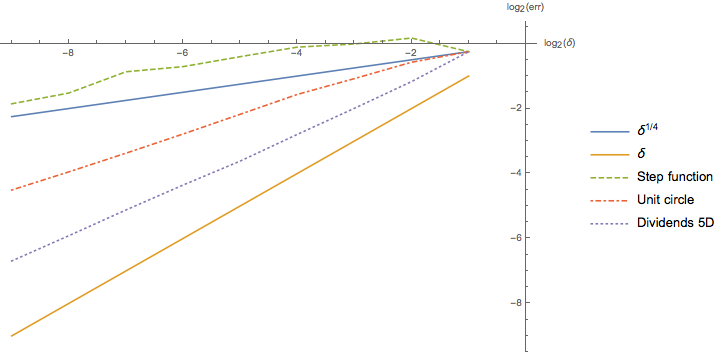}
\caption{Estimated $L^2$-errors.}\label{fig:err}
\end{center}
\end{figure}

Even though the proven rate for the Euler-Maruyama method is lower than for the transformation-based method from \cite{sz2016b}, the calculations are usually faster in practice using the first method, since the simulation of a single path is faster.
Table \ref{tab:runtimes} confirms this claim:
we observe that computation times are higher by up to two orders of magnitude for the transformation method, while the estimated error is of comparable size.

\begin{table}[ht]
\begin{center}
  \begin{tabular}{ | c | c c | c c | }
       \hline
      &
      \multicolumn{2}{c|}{computation time} &
      \multicolumn{2}{c|}{estimated error} \\ \cline{2-5}
     & EM & GM  & EM & GM \\
    \hline
   Step function & 10.84 & 86.92 & 0.1324 & 0.3362 \\ \hline
    Unit circle & 14.39 & 5267.37 & 0.0195 & 0.0323 \\ \hline
    Dividends 5D & 45.52 & 7398.97 & 0.0026 & 0.0032 \\
    \hline
    \end{tabular}
  \caption{Runtimes in seconds using sequential computation and estimated errors for the Euler-Maruyama method (EM) and the transformation method (GM) with $512$ time-steps and $1024$ paths.\label{tab:runtimes}}
  \end{center}
\end{table}

For completeness, we remark that one can construct examples, where the transformation method is much faster
while giving a smaller error.
For example, start with prescribing the transform $G(x)=x+x|x|\phi(10x)$ and set $\mu(x)=\frac{1}{2}(G^{-1})''(G(x))$ and $\sigma(x)=(G^{-1})'(G(x))$. This leads to $\tilde \mu (z)=0$ and $\tilde \sigma(z)=1$. Hence, if we use the transformation method with the same $G$, then $\appz=Z=W$ and the transformation method gives the estimate $G^{-1}(W)$, which is the exact solution.

\paragraph{Conclusion}
In this paper we have for the first time proven strong convergence and also a positive strong convergence rate for an explicit method (the Euler-Maruyama method) for multidimensional SDEs with discontinuous drift that has a degenerate diffusion coefficient, or with a discontinuous drift that does not satisfy a one-sided Lipschitz condition, or both.
The Euler-Maruyama method has the advantage that it does not need the exact form of the set of discontinuities of the drift as an input, and that in practice, computation of one path is fast in comparison to the second method in the literature that can deal with this class of SDEs.
Our numerical experiments suggest that in addition to these advantages, it even seems that the Euler-Maruyama method converges at a higher than the theoretically obtained rate for many examples and 
it will be a topic of future research to prove sharpness, or find a sharp bound.


\section*{Acknowledgements}
The authors thank Andreas Neuenkirch and Lukasz Szpruch for fruitful discussions that helped to improve the estimate in Lemma \ref{lem:omega} and hence the obtained convergence rate and Thomas M\"uller-Gronbach for pointing out an inaccuracy in the definition of $\alpha$.

G.~Leobacher is supported by the Austrian Science Fund (FWF): Project F5508-N26, which is part of the Special Research Program ``Quasi-Monte Carlo Methods: Theory and Applications''.
A part of this paper was written while G.~Leobacher was member of the Department of Financial Mathematics and Applied Number Theory, Johannes Kepler University Linz, 4040 Linz, Austria.

M.~Sz\"olgyenyi is supported by the Vienna Science and Technology Fund (WWTF): Project MA14-031.
Furthermore, M.~Sz\"olgyenyi is grateful for the AXA Research Grant “Numerical Methods for Stochastic Differential Equations with Irregular Coefficients with Applications in Risk Theory and Mathematical Finance”.



\vspace{2em}
\centerline{\underline{\hspace*{16cm}}}

 \noindent G. Leobacher \\
 Institute for Mathematics and Scientific Computing, University of Graz, Heinrichstra\ss{}e 36, 8010 Graz, Austria\\
 gunther.leobacher@uni-graz.at\\

\noindent M. Sz\"olgyenyi \Letter \\
Institute of Statistics and Mathematics, Vienna University of Economics and Business, Welthandelsplatz 1, 1020 Vienna, Austria\\
michaela.szoelgyenyi@wu.ac.at


\end{document}